\documentclass{amsart}
\usepackage{amscd,amssymb}
\usepackage{graphicx}
\usepackage[arrow,matrix]{xy}

\newcommand{\R}{\mathbb{R}}
\newcommand{\C}{\mathbb{C}}

\newcommand{\OO}{\mathcal{O}}
\newcommand{\il}{\mathcal{L}}
\newcommand{\ir}{\mathcal{R}}

\newcommand{\ik}{\mathcal{K}}
\newcommand{\jj}{\mathcal{J}}

\newtheorem{thm}{Theorem}[section]
\newtheorem{lem}{Lemma}[section]

\newtheorem{rmk}{Remark}[section]
\newtheorem{dfn}{Definition}[section]

\begin{document}

\title[Invariants of Relative Right and Contact Equivalences]{Invariants of Relative Right and Contact Equivalences}

\author[I. Ahmed]{Imran Ahmed$^{1}$}
\thanks{1 Supported by CNPq-TWAS, grant \# FR 3240188100}
\address{Imran Ahmed, Department of Mathematics, COMSATS Institute of Information Technology, M.A. Jinnah Campus, Defence Road, off Raiwind Road Lahore, PAKISTAN.}
\email{drimranahmed@ciitlahore.edu.pk}
\author[M.A.S. Ruas]{Maria Aparecida Soares Ruas$^{2}$}
\thanks{2 Partially supported by FAPESP, grant \# 08/54222-6, and CNPq,
grant \# 301474/2005-2}
\address{Maria Aparecida Soares Ruas, Departamento de Matem$\acute{a}$ticas, Instituto de Ci\^{e}ncias Matem$\acute{a}$ticas e
de Computa\c{c}\~{a}o, Universidade de S\~{a}o Paulo, Avenida Trabalhador
S\~{a}ocarlense 400, S\~{a}o
Carlos-S.P., Brazil.} \email{maasruas@icmc.usp.br}

\begin{abstract}

We study holomorphic function germs under  equivalence relations that
preserve an analytic variety.

We show that two  quasihomogeneous polynomials, not necessarily with isolated
singularities, having isomorphic relative Milnor algebras are relative right
equivalent. Under the condition that the module of vector fields tangent to the
variety is finitely generated, we also show that the relative Tjurina algebra
is a complete invariant for the classification of arbitrary function germs with
respect to the relative  contact equivalence. This is the relative version of a
well known result by Mather and Yau.
\end{abstract}



\maketitle

\theoremstyle{definition}
\newtheorem{ex}{Example}

\section{Introduction}

Let $\OO_n$ be the ring of germs of analytic functions
$h:(\C^n,0)\to\C$. Consider the analytic variety
$V=\{x:f_1(x)=\ldots=f_r(x)=0\}\subset(\C^n,0)$, where
$f_1,\ldots,f_r$ are germs of analytic functions. In this note we
study function germs $h:(\C^n,0)\to(\C,0)$ under the equivalence
relation that preserves the analytic variety $(V,0)$.

We say that two function germs $h_1$ and $h_2$ :$(\C^n,0)\to(\C,0)$
are $\ir_V$-equivalent if there exists a germ of a diffeomorphism
$\psi:(\C^n,0)\to(\C^n,0)$ with $\psi(V)=V$ and $h_1\circ\psi=h_2$.
That is,
$$\ir_V=\{\psi\in\ir:\psi(V)=V\}$$
where $\ir$ is the group of germs of diffeomorphisms of $(\C^n,0)$.

Two function germs $h_1$ and $h_2$ :$(\C^n,0)\to(\C,0)$ are
$\ik_V$-equivalent if there exists a germ of a diffeomorphism
$\psi:(\C^n,0)\to(\C^n,0)$ and a unit $u\in\OO_n^*$ such that
$\psi(V)=V$ and $h_1=u\cdot(h_2\circ\psi)$.

These equivalence relations reduce to the well known $\mathcal R$ and $\mathcal K$ equivalences when we take
$V=\{0\}.$  The investigation of  invariants of  the classification of function germs $f: (\mathbb C^n,0) \to (\mathbb C,0)$
with respect to these equivalence relations is a classical problem in singularity theory. In this setting, the Milnor algebra
 $M(f)\,=\, \frac{\OO_n}{\jj_f}$  and the Tjurina algebra  $T(f)=\frac{\OO_n}{\langle f,\jj_f\rangle}$ play a central r\^ole.  The celebrated result of  Mather and Yau in \cite{MY} gives that the Tjurina algebra
is a complete invariant for $\mathcal K$-equivalence of function germs with isolated singularities. As a consequence, it follows that the Tjurina algebras classify  hypersurfaces with isolated singularities. In a subsequent paper \cite{Y}, Yau also
considers  the right-left   equivalence ($\mathcal R \mathcal L$-equivalence) and proves that two function germs $f$ and $g$ are $\mathcal R  \mathcal L$- equivalent if and only if their Milnor algebra are mixed isomorphic. This last condition
means that there exist
$\C$-algebra isomorphisms $\sigma:M(f)\to M(g)$, $\tau:\C\{t\}/(t^{n+1})\to
\C\{t\}/(t^{n+1})$ such that $\sigma (a(t)\cdot
h)=\tau(a(t))\cdot\sigma(h)$ for any $a(t)\in\C\{t\}/(t^{n+1})$ and
any $h\in M(f)$.

One interesting question is to determine the relationship between these equivalence relations. This  was considered by Benson and Yau, in \cite{BY}, who gave   a necessary and sufficient condition for
$\ir\il$-equivalence to coincide with $\ik$-equivalence, see Theorem
5.1 \cite{BY}.

Our aim is to investigate the relative versions of these known
classical results,  without imposing the the condition of having
isolated singularities at the origin. We define the relative Milnor
and Tjurina algebra (see Definition \ref{def1}) and investigate
their r\^ole in the classification of function germs with respect to
$\mathcal R_V$ and $\mathcal K_V$  equivalences.

 In Theorem \ref{th9.2} we show that two arbitrary (i.e. not necessarily with isolated singularities) quasihomogeneous polynomials $f$ and $g$ having isomorphic relative Milnor
algebras $M_V(f)$ and $M_V(g)$ are $\ir_V$-equivalent.

The Example of Gaffney and Hauser, in \cite{GH}, suggests us that we
can not extend this result for arbitrary analytic germs.

In Theorem \ref{th9.4} we prove that two arbitrary complex-analytic
hypersurfaces (i.e. not necessary with isolated singularities), one
is quasihomogeneous and other is arbitrary, are determined by
isomorphism of Jacobian ideals.

This is the relative version of the  theorem by Mather and
Yau \cite{MY}.  For arbitrary hypersurface
singularities (i.e. not necessary with isolated singularities), the
Mather-Yau Theorem (even a more general version) has been proved by
Greuel, Lossen and Shustin \cite{GLS}.


\section{Basic definitions }

We denote by $\theta_n$ the set of germs of tangent vector fields in
$(\C^n,0)$; $\theta_n$ is a free $\OO_n$ module of rank $n$. Let
$I(V)$ be the ideal in $\OO_n$ consisting of germs of analytic
functions vanishing on $V$. We denote by
$\Theta_V=\{\eta\in\theta_n:\eta(I(V))\subseteq I(V)\}$, the
submodule of germs of vector fields tangent to $V$.

The tangent space to the action of the group $\ir_V$ is
$T\ir_V(h)=dh(\Theta_V^0)=\jj_h(\Theta_V^0)=\langle
dh(\xi_i):\,\xi_i\mbox{ are the generators of }\Theta_V^0\rangle$,
where $\Theta_V^0$ is the submodule of $\Theta_V$ given by the
vector fields that are zero at zero. When the point $x=0$ is a
stratum in the logarithmic stratification of the analytic variety,
this is the case when $V$ has an isolated singularity at the origin,
see \cite{BR} for details, both spaces $\Theta_V$ and $\Theta_V^0$
coincide.

The tangent space to the action of the group $\ik_V$ is $T\ik_V(h)=\langle h,dh(\Theta_V^0)\rangle=\langle h,\jj_h(\Theta_V^0)\rangle$.

We fix a system of local coordinates $x$ of $\C^n$. Due to the identification between $\OO_n$ and the ring of convergent power series $\C\{x_1,\ldots,x_n\}$ we
identify a germ $f\in\OO_n$ with its power series $f(x)=\sum a_{\alpha}x^{\alpha}$, where
$x^{\alpha}=x_1^{\alpha_1}\ldots x_n^{\alpha_n}$.

Now, we define relative Milnor and Tjurina algebras.

\begin{dfn}\label{def1}
The {\it relative Milnor and Tjurina algebras}, $M_V(h)$ and
$T_V(h)$, of $h$ are defined respectively, by
$$M_V(h)=\frac{\C\{x_1,\ldots,x_n\}}{\jj_h(\Theta_V^0)}\,\,\mbox{and}\,\,T_V(h)=\frac{\C\{x_1,\ldots,x_n\}}{\langle h,\jj_h(\Theta_V^0)\rangle},$$
where $\jj_h(\Theta_V^0)$ is the ideal defined by
$$\jj_h(\Theta_V^0)=dh(\Theta_V^0)=\langle
dh(\xi_i):\,\xi_i\mbox{ are the generators of }\Theta_V^0\rangle.$$
\end{dfn}

When $V$ is a weighted homogeneous variety, we can always choose
weighted homogeneous generators for $\Theta_V$. Moreover, it is
finitely generated \cite{D}, Lemma 3.2, p.41. In this case, the
ideal $\jj_h(\Theta_V^0)$ is given by
$$\jj_{h}(\Theta_V^0)=\langle dh(\xi_i):i=1,\ldots,p\rangle$$
where $dh(\xi_i)=\sum_{j=1}^n a_{ij}x^{\underline{P}}\frac{\partial
h}{\partial x_j},\,\,<\underline{P},\underline{w}>=w_j$.

\section{Quasihomogeneous Functions and Filtrations}

We recall first some basic facts on quasihomogeneous functions and filtrations in the ring $A$ of formal power series. We introduce, in the next section, their analogues for quasihomogeneous diffeomorphisms and vector fields. For a more complete introduction see \cite{A1}, Chap. 1, \S 3.

A holomorphic function $f:(\C^n,0)\to (\C,0)$ (defined on the complex space $\C^n$) is a quasihomogeneous function of
degree $d$ with weights $w_1,\ldots, w_n$ if
$$f(\lambda^{w_1}x_1,\ldots,\lambda^{w_n}x_n)=\lambda^d f(x_1,\ldots,x_n)\,\forall\, \lambda >0.$$

In terms of the Taylor series $\sum f_{\underline{k}}x^{\underline{k}}$ of $f$, the quasihomogeneity
condition means that the exponents of the nonzero terms of the series lie in the hyperplane
$$L=\{\underline{k}:w_1k_1+\ldots+w_n k_n=d\}.$$

Any quasihomogeneous function $f$ of degree $d$ satisfies Euler's identity
\begin{equation}\label{euler}
\sum_{i=1}^n w_i x_i \frac{\partial f}{\partial x_i}=d.f
\end{equation}
It implies that a quasihomogeneous function $f$ belongs to its
Jacobian ideal $\jj_f$. The following is the well known result of
Saito \cite {S}.
\begin{thm}
A function-germ $f:(\C^n,0)\to (\C,0)$ is equivalent to a
quasihomogeneous function-germ if and only if $f\in\jj_f$.
\end{thm}

Consider $\C^n$ with a fixed coordinate system $x_1,\ldots,x_n$. The
algebra of formal power series in the coordinates will be denoted by
$A=\C[[x_1,\ldots,x_n]]$. We assume that a quasihomogeneity type
$\underline{w}=(w_1,\ldots,w_n)$ is fixed. With each such
$\underline{w}$ there is associated a filtration of the ring $A$,
defined as follows.

The monomial ${\bf x}^{\underline{k}}$ is said to have degree $d$ if
$<\underline{w},\underline{k}>=w_1k_1+\ldots+w_n k_n=d$.

The order $d$ of a series (resp. polynomial) is the smallest of the
degrees of the monomials that appear in that series (resp.
polynomial).

The series of order larger than or equal to $d$ form a subspace
$A_d\subset A$. The order of a product is equal to the sum of the
orders of the factors. Consequently, $A_d$ is an ideal in the ring
$A$. The family of ideals $A_d$ constitutes a decreasing filtration
of $A$: $A_{\acute{d}}\subset A_d$ whenever $\acute{d}>d$. We let
$A_{d+}$ denote the ideal in $A$ formed by the series of order
higher than $d$.

The quotient algebra $A/A_{d+}$ is called the algebra of
$d$-quasijets, and its elements are called $d$-quasijets.

\section{Quasihomogeneous Diffeomorphisms and Vector Fields}

Several Lie groups and algebras are associated with the filtration
defined in the ring $A$ of power series by the type of
quasihomogeneity $\underline{w}$. In the case of ordinary
homogeneity these are the general linear group, the group of
$k$-jets of diffeomorphisms, its subgroup of $k$-jets with
$(k-1)$-jet equal to the identity, and their quotient groups. Their
analogues for the case of a quasihomogeneous filtration are defined
as follows.

A formal diffeomorphism $g:(\C^n,0)\to (\C^n,0)$ is a set of $n$
power series $g_i\in A$ without constant terms for which the map
$g^{\ast}:A\to A$ given by the rule $g^{\ast}f=f\circ g$ is an
algebra isomorphism.

The diffeomorphism $g$ is said to have order $d$ if for every $s$
$$(g^{\ast}-1)A_s\subset A_{s+d}.$$

The set of all diffeomorphisms of order $d\geq0$ is a group $G_d$.
The family of groups $G_d$ yields a decreasing filtration of the
group $G$ of formal diffeomorphisms; indeed, for $\acute{d}>d\geq0$,
$G_{\acute{d}}\subset G_d$ and is a normal subgroup in $G_d$.

The group $G_0$ plays the role in the quasihomogeneous case that the
full group of formal diffeomorphisms plays in the homogeneous case.
We should emphasize that in the quasihomogeneous case $G_0\neq G$,
since certain diffeomorphisms have negative orders and do not belong
to $G_0$.

The group of $d$-quasijets of type $\underline{w}$ is the quotient
group of the group of diffeomorphisms $G_0$ by the subgroup $G_{d+}$
of diffeomorphisms of order higher than $d$: $J_d=G_0/G_{d+}$.

Note that in the ordinary homogeneous case our numbering differs from the
standard one by $1$: for us $J_0$ is the group of $1$-jets and so
on.

$J_d$ acts as a group of linear transformations on the space $A/A_{d+}$ of $d$-quasijets of functions. A special importance is attached to the group $J_0$, which is the quasihomogeneous generalization of the general linear group.

A diffeomorphism $g\in G_0$ is said to be quasihomogeneous of type
$\underline{w}$ if each of the spaces of quasihomogeneous functions
of degree $d$ (and type \underline{w}) is invariant under the action
of $g^{\ast}$.

The set of all quasihomogeneous diffeomorphisms is a subgroup of
$G_0$. This subgroup is canonically isomorphic to $J_0$, the
isomorphism being provided by the restriction of the canonical
projection $G_0\to J_0$.

\medskip

The infinitesimal analogues of the concepts introduced above look as
follows.

A formal vector field $v=\sum v_i\partial_i$, where
$\partial_i=\partial/\partial x_i$, is said to have order $d$ if
differentiation in the direction of $v$ raises the degree of any
function by at least $d$: $L_v A_s\subset A_{s+d}$.

We let $\mathfrak{g}_d$ denote the set of all vector fields of order $d$. The
filtration arising in this way in the Lie algebra $\mathfrak{g}$ of vector
fields (i.e., of derivations of the algebra $A$) is compatible with
the filtrations in $A$ and in the group of diffeomorphisms $G$:\\
1. $f\in A_d, v\in \mathfrak{g}_s\Rightarrow fv\in \mathfrak{g}_{d+s}, L_vf\in A_{d+s}$\\
2. The module $\mathfrak{g}_d$, $d\geq 0$, is a Lie algebra w.r.t. the Poisson
bracket of vector fields.\\
3. The Lie algebra $\mathfrak{g}_d$ is an ideal in the Lie algebra $\mathfrak{g}_0$.\\
4. The Lie algebra $\mathfrak{j}_d$ of the Lie group $J_d$ of $d$-quasijets of
diffeomorphisms is equal to the quotient algebra $\mathfrak{g}_0/\mathfrak{g}_{d+}$.\\
5. The quasihomogeneous vector fields of degree $0$ form a finite-dimensional Lie subalgebra of the Lie algebra $\mathfrak{g}_0$; this subalgebra
is canonically isomorphic to the Lie algebra $\mathfrak{j}_0$ of the group of
$0$-jets of diffeomorphisms.

The support of a quasihomogeneous function of degree $d$ and type $\underline{w}$ is the set of all points $\underline{k}$ with nonnegative integer coordinates on the diagonal $$L=\{\underline{k}:\langle\underline{k},\underline{w}\rangle=d\}.$$

Quasihomogeneous functions can be regarded as functions given on their supports: $\sum f_{\underline{k}}x^{\underline{k}}$ assumes at $\underline{k}$ the value $f_{\underline{k}}$. The set of all such functions is a linear space $\C^r$, where $r$ is the number of points in the support. Both the group of quasihomogeneous diffeomorphisms (of type \underline{w}) and its Lie algebra $\mathfrak{a}$ act on this space.

The Lie algebra $\mathfrak{a}$ of a quasihomogeneous vector field of degree $0$
is spanned, as a $\C$-linear space, by all monomial fields
$x^{\underline{P}}\partial_i$ for which
$<\underline{P},\underline{w}>=w_i$. For example, the $n$ fields
$x_i\partial_i$ belong to $\mathfrak{a}$ for any $\underline{w}$.

\begin{ex}
Consider the quasihomogeneous polynomial $f=x^2y+z^2$ of degree
$d=6$ w.r.t. weights $(2,2,3)$. Note that the Lie algebra of
quasihomogeneous vector fields of degree $0$ is spanned by
$$\mathfrak{a}=\langle x^{\underline{P}}\partial_i:<\underline{P},\underline{w}>=w_i, i=1,2,3 \rangle
=\langle x\frac{\partial}{\partial x},x\frac{\partial}{\partial
y},y\frac{\partial}{\partial x}, y\frac{\partial}{\partial
y},z\frac{\partial}{\partial z}\rangle$$
\end{ex}

\section{Function Germs with Isomorphic Local Algebras}

We recall first Mather's lemma providing effective necessary  and sufficient conditions for a connected submanifold (in our case the path $P$) to be contained in an orbit.

\begin{lem} (\cite{Ma})\label{lem4.3.1}
Let $m:G\times M\to M$ be a smooth action and $P\subset M$ a
connected smooth submanifold. Then $P$ is contained in a single
$G$-orbit if and only if the following conditions are
fulfilled:\\
(a) $T_x(G\cdot x)\supset T_xP$, for any $x \in P$.\\
(b) $\dim T_x(G\cdot x)$ is constant for $x\in P$.
\end{lem}

For arbitrary (i.e. not necessary with isolated singularities)
quasihomogeneous polynomials we establish the following results.

\begin{lem}\label{lem9.1}
Let $f,g\in H_{\underline{w}}^d(n,1;\C)=H_{\underline{w}}^d$ be two
quasihomogeneous polynomials of degree $d$ and $\Phi\in H_{\underline{w}}^r(n,1;\C)$ be a quasihomogeneous polynomial of degree $r$ w.r.t. the same weights $\underline{w}=(w_1,\ldots,w_n)$ such that $\jj_f(\Theta_V^0)=\jj_g(\Theta_V^0)$, where $\Phi^{-1}(0)=V$ is a hypersurface in $(\C^n,0)$. Then
$f\stackrel{\ir_V}{\sim}g$, where $\stackrel{\ir_V}{\sim}$ denotes the relative right equivalence.
\end{lem}

\begin{proof}

To prove this claim choose an appropriate submanifold of
$H_{\underline{w}}^d(n,1;\C)$ containing $f$ and $g$ and then apply
Mather's lemma to get the result.

Let $f,g\in H_{\underline{w}}^d(n,1;\C)$ such that
$\jj_f(\Theta_V^0)=\jj_g(\Theta_V^0)$. Set $f_t=(1-t)f+tg\in
H_{\underline{w}}^d(n,1;\C)$. Consider the $\ir_V$-equivalence
action on $H_{\underline{w}}^d(n,1;\C)$ under the group
$\ir_V^0=\ir_V\cap J_0$, we have
\begin{equation}\label{eq9.1}
T_{f_t}(\ir_V^0\cdot f_t)=\jj_{f_t}(\Theta_V^0)\cap
H_{\underline{w}}^d=\langle df_t(\xi_i):i=1,\ldots,p\rangle\cap
H_{\underline{w}}^d\subset T_{f_t}(J_0\cdot f_t)
\end{equation}
where $df_t(\xi_i)=\sum_{j=1}^n a_{ij}x^{\underline{P}}\frac{\partial f_t}{\partial
x_j}=\sum_{j=1}^n a_{ij}x^{\underline{P}}[(1-t)\frac{\partial f}{\partial
x_j}+t\frac{\partial g}{\partial
x_j}],\,\,<\underline{P},\underline{w}>=w_j$.\\
Note that we have the inclusion of finite dimensional $\C$-vector spaces
\begin{equation}\label{eq9.2}
T_{f_t}(\ir_V^0\cdot f_t)=\langle df_t(\xi_i)\rangle\cap
H_{\underline{w}}^d\subset\jj_f(\Theta_V^0)\cap H_{\underline{w}}^d
\end{equation}
with equality for $t=0$ and $t=1$. The spaces $\jj_{f_t}(\Theta_V^0)\cap H_{\underline{w}}^d$ and $\jj_f(\Theta_V^0)\cap H_{\underline{w}}^d$ are not trivial by Euler identity \ref{euler}.

Let's  show that we have equality for all $t\in[0,1]$ except finitely many values.\\
Take $\dim(\jj_{f_t}(\Theta_V^0)\cap H_{\underline{w}}^d)=\dim(\jj_f(\Theta_V^0)\cap H_{\underline{w}}^d)=s$ (say).
Let's fix $\{e_1,\ldots,e_s\}$ a basis of $\jj_f(\Theta_V^0)\cap H_{\underline{w}}^d$. Consider the $s$ polynomials corresponding to the generators of the space
\eqref{eq9.1}:
\[\alpha_i(t)=df_t(\xi_i)=\sum_{j=1}^n a_{ij}x^{\underline{P}}\frac{\partial f_t}{\partial
x_j}=\sum_{j=1}^n a_{ij}x^{\underline{P}}[(1-t)\frac{\partial f}{\partial
x_j}+t\frac{\partial g}{\partial
x_j}],\,\,<\underline{P},\underline{w}>=w_j\]
We can express each $\alpha_i(t)$, $i=1,\ldots,s$ in terms of
above mentioned fixed basis as
\begin{equation}\label{eq9.3}
\alpha_i(t)=\phi_{i1}(t)e_1+\ldots+\phi_{is}(t)e_s,\,\,\forall\,\,
i=1,\ldots,s
\end{equation}
where each $\phi_{ij}(t)$ is linear in $t$. Consider the matrix of
transformation corresponding to the eqs. \eqref{eq9.3}
\[
(\phi_{ij}(t))_{s\times s}= \left(
  \begin{array}{cccc}
    \phi_{11}(t) & \phi_{12}(t)& \ldots & \phi_{1s}(t) \\
    \vdots & \vdots & \ddots & \vdots \\
    \phi_{s1}(t) & \phi_{s2}(t) & \ldots & \phi_{ss}(t) \\
  \end{array}
\right)
\]
having rank at most $s$. Note that the equality $$\jj_{f_t}(\Theta_V^0)\cap H_{\underline{w}}^d=\jj_f(\Theta_V^0)\cap H_{\underline{w}}^d$$ holds for those values of $t$ in $\C$ for
which the rank of above matrix is precisely $s$. We have the $s\times
s$-matrix whose determinant is a polynomial of degree $s$ in $t$
and by the fundamental theorem of algebra it has at most $s$ roots
in $\C$ for which rank of the matrix of transformation will be less
than $s$. Therefore, the above-mentioned equality does not hold for
at most finitely many values, say $t_1,\ldots,t_q$ where
$1\leq q\leq s$.

It follows that the dimension of the space \eqref{eq9.1} is constant
for all $t\in \C$ except finitely many values $\{t_1,\ldots,t_q\}$.\\
For an arbitrary smooth path
\[\alpha:\C\longrightarrow\C\backslash \{t_1,\ldots,t_q\}\]
with $\alpha(0)=0$ and $\alpha(1)=1$, we have the connected smooth
submanifold
\[P=\{f_t=(1-\alpha(t))f(x)+\alpha(t)g(x):\,t\in\C\}\]
of $H_{\underline{w}}^d$. By the above, it follows $\dim
T_{f_t}(\ir_V^0\cdot f_t)$ is constant for $f_t\in P$.

Now, to apply Mather's lemma, we need to show that the tangent space
to the submanifold $P$ is contained in that to the orbit
$\ir_V^0\cdot f_t$ for any $f_t\in P$. One clearly has
\[T_{f_t}P=\{\dot{f_t}=-\dot{\alpha}(t)f(x)+\dot{\alpha}(t)g(x):\,\forall\,t\in\C\}\]
Therefore, by Euler formula \ref{euler}, we have\[T_{f_t}P\subset
T_{f_t}(\ir_V^0\cdot f_t)\] By Mather's lemma the submanifold $P$ is
contained in a single orbit. Hence the result.
\end{proof}

\begin{thm}\label{th9.2}
Let $f,g\in H_{\underline{w}}^d(n,1;\C)=H_{\underline{w}}^d$ be two
quasihomogeneous polynomials of degree $d$ and $\Phi\in H_{\underline{w}}^r(n,1;\C)$ be a quasihomogeneous polynomial of degree $r$ w.r.t. the same weights $\underline{w}=(w_1,\ldots,w_n)$. If $M_V(f)\simeq M_V(g)$
(isomorphism of graded $\C$-algebra) then $f\stackrel{\ir_V}{\sim}g$, where $\Phi^{-1}(0)=V$ is a hypersurface in $(\C^n,0)$.
\end{thm}

\begin{proof}

We show firstly that an isomorphism of graded $\C$-algebras
\[\varphi:(M_V(g))_l=(\frac{\C[x_1,\ldots,x_n]}{\jj_g(\Theta_V^0)})_l\stackrel{\simeq}{\longrightarrow}
(M_V(f))_l=(\frac{\C[x_1,\ldots,x_n]}{\jj_f(\Theta_V^0)})_l\] is induced by an isomorphism
$u:\C^n\longrightarrow\C^n$ such that $u^*(\jj_g(\Theta_V^0))=\jj_f(\Theta_V^0)$.

Consider the following commutative diagram.
$$\xymatrix{
0 \ar[d] & 0 \ar[d] \\
(\jj_g(\Theta_V^0))_{d+l}  \ar@{.>}[r]^{u^*} \ar[d]^i & (\jj_f(\Theta_V^0))_{d+l} \ar[d]^j \\
J_{\underline{w}}^{d+l}  \ar@{.>}[r]^{u^*} \ar[d]^p & J_{\underline{w}}^{d+l}
\ar[d]^q
\\
(M_V(g))_l \ar[r]^{\varphi}_{\simeq} \ar[d] & (M_V(f))_l \ar[d] \\
0 & 0 }
$$
Define the morphism $u^*:J_{\underline{w}}^{d+l}\rightarrow
J_{\underline{w}}^{d+l}$ by
\begin{equation}\label{eq9.4}
u^*(x_i)=L_i(x_1,\ldots,x_n)=\sum_{j=1}^n a_{ij}x_j^{\alpha_j}+\sum
a_{ik_1\ldots k_n}x_{k_1}^{\beta_1}\ldots
x_{k_n}^{\beta_n}\,;\,i=1,\ldots,n
\end{equation}
where
$k_m\in\{1,\ldots,n\}\,\&\,w_{k_1}\beta_1+\ldots+w_{k_n}\beta_n=deg_{\underline{w}}(x_i)=w_j\alpha_j$,
which is well defined by commutativity of diagram below.
$$\xymatrix{
x_i  \ar@{|.>}[r]^{u^*} \ar@{|->}[d]^p &  L_i\ar@{|->}[d]^q \\
\widehat{x_i} \ar@{|->}[r]^{\varphi}_{\simeq} & \widehat{L_i}}$$
Note that the isomorphism $\varphi$ is a degree preserving map and
is also given by the same morphism $u^*$. Therefore, $u^*$ is an
isomorphism.\\
Now, we show that $u^*(\jj_g(\Theta_V^0))=\jj_f(\Theta_V^0)$. For every $G\in(\jj_g(\Theta_V^0))_{d+l}$, we
have $u^*(G)\in(\jj_f(\Theta_V^0))_{d+l}$ by commutative diagram below.
$$\xymatrix{
G  \ar@{|->}[r]^{u^*} \ar@{|->}[d]^p &  F=u^*(G) \ar@{|->}[d]^q \\
\widehat{0} \ar@{|->}[r]^{\varphi} & \widehat{F}=\widehat{0}}
$$
It implies that $u^*((\jj_g(\Theta_V^0))_{d+l})\subset(\jj_f(\Theta_V^0))_{d+l}$. As $u^*$ is an
isomorphism, therefore it is invertible and by repeating the above
argument for its inverse, we have $u^*((\jj_g(\Theta_V^0))_{d+l})\supset(\jj_f(\Theta_V^0))_{d+l}$.
Therefore, $u^*((\jj_g(\Theta_V^0))_{d+l})=(\jj_f(\Theta_V^0))_{d+l}$. It follows that $u^*(\jj_g(\Theta_V^0))=\jj_f(\Theta_V^0)$.
Thus, $u^*$ is an isomorphism with $u^*(\jj_g(\Theta_V^0))=\jj_f(\Theta_V^0)$.

By eq. \eqref{eq9.4}, the map $u:\C^n\to\C^n$ can be defined by
$$u(z_1,\ldots,z_n)=(L_1(z_1,\ldots,z_n),\ldots,L_n(z_1,\ldots,z_n))$$
where $L_i(z_1,\ldots,z_n)=\sum_{j=1}^n a_{ij}x_j^{\alpha_j}+\sum
a_{ik_1\ldots k_n}x_{k_1}^{\beta_1}\ldots
x_{k_n}^{\beta_n}\,;\,i=1,\ldots,n, k_m\in\{1,\ldots,n\}\\\&\,w_{k_1}\beta_1+\ldots+w_{k_n}\beta_n=deg_{\underline{w}}(x_i)=w_j\alpha_j$. Note that $u$ is an isomorphism by Prop. 3.16 \cite{D1}, p.23.

In this way, we have shown that the isomorphism $\varphi$ is induced
by the isomorphism $u:\C^n\to\C^n$ such that
$u^*(\jj_g(\Theta_V^0))=\jj_f(\Theta_V^0)$.

Consider $u^*(\jj_g(\Theta_V^0))=<g_1\circ u,\ldots,g_n\circ u>=\jj_{g\circ u}(\Theta_V^0)$,
where $g_j$ are the generators of $\jj_g(\Theta_V^0)$. Therefore, $\jj_{g\circ
u}(\Theta_V^0)=\jj_f(\Theta_V^0)\Rightarrow g\circ u\stackrel{\ir_V}{\sim}f$, by Lemma
\ref{lem9.1}. Hence, by definition there exists an analytic isomorphism $h\in \ir_V$ such that $g\circ u=f\circ h$. Since $\ir_V$ is a group, therefore $h^{-1}\in \ir_V$. Taking $u=h^{-1}$ we have
$g\circ h\stackrel{\ir_V}{\sim}g$. Thus, $ f\stackrel{\ir_V}{\sim}g$.
\end{proof}

\begin{rmk}\label{rk9.3}
The converse implication, namely
$$f\stackrel{\ir_V}{\sim}g\Rightarrow M_V(f)\simeq M_V(g)$$
always holds(even for analytic germs $f,\,g$ defining IHS).
\end{rmk}

\begin{proof}
Let $f\stackrel{\ir_V}{\sim}g$. Then, by definition, there exists an
analytic isomorphism $h\in \ir_V$ such that $f\circ h=g$. It follows
that $\jj_{f\circ
h}(\Theta_V^0)=\jj_g(\Theta_V^0)\Rightarrow h^*(\jj_f(\Theta_V^0))=\jj_g(\Theta_V^0)$.\\
By Prop. 3.16 \cite{D1}, p.23 $h^*$ is also analytic isomorphism.\\
Thus, $M_V(f)\cong M_V(g)$ by the commutativity of the diagram below.
$$\xymatrix{
0 \ar[d] & 0 \ar[d] \\
\jj_g(\Theta_V^0)  \ar[r]^{h^*} \ar[d]^i & \jj_f(\Theta_V^0) \ar[d]^j \\
\C[x_1,\ldots,x_n]  \ar[r]^{h^*} \ar[d]^p & \C[x_1,\ldots,x_n]
\ar[d]^q
\\
M_V(g) \ar[r]^{\varphi} \ar[d] & M_V(f) \ar[d] \\
0 & 0 }
$$
\end{proof}

\begin{thm}\label{th9.4}
Let $f\in H_{\underline{w}}^d(n,1;\C)=H_{\underline{w}}^d$ be a
quasihomogeneous polynomials of degree $d$ and $\Phi\in H_{\underline{w}}^r(n,1;\C)$ be a quasihomogeneous polynomial of degree $r$ w.r.t. the same weights $\underline{w}=(w_1,\ldots,w_n)$. Let $g$ be an arbitrary analytic germ  such that $\jj_f(\Theta_V^0)\cong\jj_g(\Theta_V^0)$, where $\Phi^{-1}(0)=V$ is a hypersurface in $(\C^n,0)$. Then
$f\stackrel{\ir_V}{\sim}g$.
\end{thm}

\begin{proof}
Let $\jj_f(\Theta_V^0)\cong\jj_g(\Theta_V^0)$. Then there exists an analytic isomorphism $h\in \ir$ such that $h^*(\jj_g(\Theta_V^0))=\jj_f(\Theta_V^0)$ by Prop. 3.16 \cite{D1}, p.23.\\
It follows that $\jj_{g\circ
h}(\Theta_V^0)=\jj_f(\Theta_V^0)$, where $g\circ h$ is quasihomogeneous polynomial. It implies that $g\circ h\stackrel{\ir_V}{\sim}f$ by Lemma \ref{lem9.1}. Hence, by definition there exists an analytic isomorphism $u\in \ir_V$ such that $g\circ h=f\circ u$. Since $\ir_V$ is a group, therefore $u^{-1}\in \ir_V$. Taking $h=u^{-1}$ we have
$g\circ h\stackrel{\ir_V}{\sim}g$. Thus, $ f\stackrel{\ir_V}{\sim}g$.
\end{proof}

The following Example of Gaffney and Hauser, in \cite{GH}, suggests
us that we can not extend the Lemma \ref{lem9.1} and Theorem \ref{th9.4} for arbitrary
analytic germs.

\begin{ex}
Let $h:(\C^n,0)\to (\C,0)$ be any function satisfying $h\notin
\jj_h\subseteq\OO_n$ i.e. $h\notin H_{\underline{w}}^d(n,1;\C)$.
Define a family $f_t:(\C^n\times\C^n\times\C,0)\to(\C,0)$ by
$f_t(x,y,z)=h(x)+(1+z+t)h(y)$, and let
$(X_t,0)\subseteq(\C^{2n+1},0)$ be the hypersurface defined by
$f_t$. Note that
$$\jj_{f_t}=\langle\frac{\partial h}{\partial x_i}(x),\frac{\partial
h}{\partial y_j}(y),h(y)\rangle,\,\,t\in\C.$$ On the other hand, the
family $\{(X_t,0)\}_{t\in\C}$ is not trivial i.e.
$(X_t,0)\ncong(X_0,0)$: For, if $\{f_t\}_{t\in\C}$ were trivial, we
would have by Prop. 2, $\S 1$, \cite{GH}$$\frac{\partial
f_t}{\partial t}=h(y)\in
(f_t)+m_{2n+1}\jj_{f_t}=(f_t)+m_{2n+1}\jj_{h(x)}+m_{2n+1}\jj_{h(y)}+m_{2n+1}(h(y))$$
Solving for $h(y)$ implies either $h(y)\in \jj_{h(y)}$ or $h(x)\in
\jj_{h(x)}$ contradicting the assumption on $h$.\\
It follows that $f_t$ is not $\ir$-equivalent to $f_0$.
\end{ex}

Before proceeding further, we state the lifting lemma.
\begin{lem}\label{lem3} Let $\varphi$ be a morphism of analytic $K$-algebras
$$\varphi:A=K\langle x_1,\ldots,x_n\rangle/I\to B=K\langle y_1,\ldots,y_m\rangle/J$$
Then $\varphi$ has a lifting $\widetilde{\varphi}:K\langle {\bf x}\rangle\to K\langle {\bf y}\rangle$ which can be chosen as an isomorphism in the case that $\varphi$ is an isomorphism and $n=m$.
\end{lem}

For arbitrary hypersurface singularities (i.e. not necessary with isolated singularities), the following result has been obtained by  Greuel, Lossen and Shustin \cite{GLS}, 2007.

\begin{thm}(Mather-Yau Theorem) Let $f,g\in\C\{x_1,\ldots,x_n\}$ be two arbitrary hypersurface
singularities having isomorphic Tjurina algebras $T(f) \simeq T(g)$.
Then $f\stackrel{\ik}{\sim}g$, where $\stackrel{\ik}{\sim}$ denotes
the contact equivalence.
\end{thm}

For arbitrary hypersurface singularities (i.e. not necessary with
isolated singularities), we establish now the relative version of
Mather-Yau Theorem under the hypothesis that $V=\Phi^{-1}(0)$ be a
variety in $(\C^n,0)$ such that $\Phi\in\C\{x_1,\ldots,x_n\}$ and
$\Theta_V$ is finitely generated. We remark that if
$\Phi\in\C[[x_1,\ldots,x_n]]$, the ring of formal power series, it
is known that $\Theta_V$ is always finitely generated \cite{GS}, Th.
5.4, p.771. However, we do not know whether the result holds for all
$\Phi\in\C\{x_1,\ldots,x_n\}$.

\begin{thm}\label{th9.5}
Let $f,g\in\C\{x_1,\ldots,x_n\}$ be two arbitrary hypersurface
singularities having isomorphic relative Tjurina algebras
$T_V(f)\simeq T_V(g)$, where $V=\Phi^{-1}(0)$ be a variety in
$(\C^n,0)$ such that $\Phi\in\C\{x_1,\ldots,x_n\}$ and $\Theta_V$ is
finitely generated. Then $f\stackrel{\ik_V}{\sim}g$, where
$\stackrel{\ik_V}{\sim}$ denotes the relative contact equivalence.
\end{thm}

\begin{proof}
Consider the isomorphism of graded $\C$-algebras
\[\varphi:T_V(f)=\frac{\C\{ x_1,\ldots,x_n\}}{\langle f,\jj_f(\Theta_V^0)\rangle}\stackrel{\simeq}{\longrightarrow}
T_V(g)=\frac{\C\{ x_1,\ldots,x_n\}}{\langle g,\jj_g(\Theta_V^0)\rangle}.\]
Note that $\varphi$ lifts to an isomorphism
$\widetilde{\varphi}:\C\{ x_1,\ldots,x_n\}\to \C\{ x_1,\ldots,x_n\}$ with $\widetilde{\varphi}(\langle f,\jj_f(\Theta_V^0)\rangle)=\langle g,\jj_g(\Theta_V^0)\rangle$
by Lifting Lemma \ref{lem3}. Since $\widetilde{\varphi}(\langle f,\jj_f(\Theta_V^0)\rangle)=\langle \widetilde{\varphi}(f),\jj_{\widetilde{\varphi}(f)}(\Theta_V^0)\rangle$, we may assume that
\begin{equation}\label{eq1}
\langle f,\jj_f(\Theta_V^0)\rangle=\langle g,\jj_g(\Theta_V^0)\rangle.
\end{equation}
Let $f,g\in\C\{x_1,\ldots,x_n\}$. Set $F(x,t)=f_t(x)=(1-t)f(x)+tg(x)$ for $t\in\C$. Thus $(f_t)$ is a 1-parameter family of germs with $f_0=f$, $f_1=g$. We intend to show that any two germs in this family satisfying Equation \ref{eq1} are $\ik_V$ equivalent.\\
Consider the family of ideals
$$I_{f_t}=\langle f_t,\jj_{f_t}(\Theta_V^0)\rangle\subset\C\{x_1,\ldots,x_n,t\},\,t\in\C.$$
By Eq. \ref{eq1}, $I_{f_t}\subset I_f=\langle f,\jj_f(\Theta_V^0)\rangle$ and $I_f=I_g$.\\
Now, represent $f,g$ in a neighbourhood $W=W({\bf 0})\subset\C^n$ by holomorphic functions and consider the coherent $\OO_{W\times\C}$-module
$$\mathfrak{F}=\langle f,\jj_f(\Theta_V^0)\rangle/\langle f_t,\jj_{f_t}(\Theta_V^0)\rangle,$$
whose support is a closed analytic set in $W\times\C$, see A.7 \cite{GLS}. Moreover, note that
$$supp(\mathfrak{F})\cap(\{0\}\times\C)=\{t\in\C|\mathfrak{F}_{(0,t)}\neq0\}=\{t\in\C|I_f\neq I_{f_t}\},$$
which is a closed analytic, hence a discrete, set of points in $\C=\{{\bf 0}\}\times\C$. It follows that the set $U=\{t\in\C|I_{f_t}=I_f\}$ is open and connected and contains 0 and 1. Note that
$$\frac{\partial f_t}{\partial t}=g-f\in I_f=I_{f_t}=\langle f_t,\jj_{f_t}(\Theta_V^0)\rangle$$
for all $t\in U$. By Theorem 2.22 \cite{GLS}, p.126, we get that $f_t\stackrel{\ik_V}{\sim}f_{\acute{t}}$ for $t,\acute{t}\in U$ such that $|t-\acute{t}|$ is sufficiently small. Therefore, $f_t\stackrel{\ik_V}{\sim}f$ for all $t\in U$, in particular, $f\stackrel{\ik_V}{\sim}g$.
\end{proof}

Note that the converse implication is just an application of the chain rule, as performed in the proof of Lemma 2.10 \cite{GLS}, p.119.

The Theorems \ref{th9.2} and \ref{th9.4} are particular cases of Theorem \ref{th9.5}. We dealt with the quasihomogeneous part first to explore the ideas deeply.

\end{document}